\documentclass[11pt,reqno]{amsart}
\usepackage{amsmath,amsthm,amssymb,enumerate}
\usepackage{hyperref}
\usepackage{mathrsfs}

\usepackage{color}

\newtheorem{theorem}{Theorem}[section]
\newtheorem{lemma}[theorem]{Lemma}

\theoremstyle{definition}
\newtheorem{assumption}{Assumption}[section]
\newtheorem{definition}{Definition}[section]

\theoremstyle{remark}
\newtheorem{remark}{Remark}[section]

\newcommand\bR{\mathbb{R}}

\newcommand\bB{\mathbb{B}}

\newcommand{\bZ}{\mathbb{Z}}
\newcommand{\bA}{\mathbb{A}}

\newcommand*{\beq}{\begin{equation}}
\newcommand*{\eeq}{\end{equation}}

\newcommand{\bit}{\begin{itemize}}
\newcommand{\eit}{\end{itemize}}
\newcommand\D{\partial}

\begin{document}

\title[Finite differences for PIDEs]{On Finite difference schemes for partial  integro-differential equations of L\'evy type}

\author[K. Dareiotis]{Konstantinos Dareiotis
}
\address{Department of Mathematics, Uppsala University, Box 480, 
751 06 Uppsala, Sweden}
\email{konstantinos.dareiotis@math.uu.se}


\begin{abstract}
In this article we introduce a finite difference  approximation for integro-differential operators of L\'evy type. We approximate solutions of integro-differential equations, where the second order operator is allowed to degenerate. In  the existing literature,  the L\'evy operator is treated as a zero/first  order operator outside of  a centered  ball of radius $\delta$, leading to error estimates of order $\xi (\delta)+N(\delta)(h+\sqrt{\tau})$,  where $h$ and $\tau$ are the spatial and temporal discretization parameters respectively. In these estimates $\xi (\delta) \downarrow 0$, but   $N(\delta )\uparrow \infty$ as $\delta \downarrow 0$. In contrast, we treat the integro-differential operator as a second order operator on the whole unit ball. By this method  we obtain error estimates of order $(h+\tau^k)$ for $k\in \{1/2,1\}$, eliminating the additional errors and  the blowing up constants.  Moreover,  we do not pose any conditions on the L\'evy measure.
\end{abstract}

\maketitle


\section{Introduction}

In the present article we consider a finite difference approximation scheme  for the partial integro-differential equation (PIDE)
\begin{align}          \label{eq: main equation}
du_t(x) &= [(L_t+J) u_t(x) +f_t(x) ] \ dt, \ (t,x) \in [0,T]\times \bR, \\              \label{eq: main equation initial condition}
u_0(x) & = \psi (x), \ x \in \bR ,     
\end{align}
where the operators are given by 
\begin{align*}
L_t \phi (x) &= a_t(x)\D_x^2 \phi(x)+b_t(x)\D_x\phi(x)+c_t(x)\phi(x), \\
J\phi(x) &= \int_{\bR} \left( \phi(x+z)-\phi(x)-I_{|z|\leq 1 } z \D_x \phi(x)  \right)\nu(dz).
\end{align*}
and the coefficient of the second derivative in $L_t$ is allowed to degenerate. 
Here $\nu$ denotes a L\'evy measure on $\bR$, that is a Borel measure on $\bR$  such that
$$
\nu(\{0\}) =0, \ \int_{\bR} 1\wedge z^2 \nu(dz) < \infty.
$$
Equations of this form are of importance, since are satisfied by certain functionals  of  jump-diffusion Markov processes, that are known to be of interest  in mathematical finance (for further reading on the subject we refer to \cite{CT}).

Finite difference schemes for equations of this form have previously been studied in \cite{CV}, \cite{DeF} and \cite{MIJ}. In these articles the integro-differential operator  is either truncated, or approximated by a second order 
difference operator in a neighborhood around the origin of radius $\delta>0$,   and the remaining operator (the integral over  $\{|z|\geq \delta \}$) is treated as a zero/first order operator.   In \cite{CV}, the solution $u$  is first approximated by $u^\delta$, the solution of the corresponding equation where the integral operator over $\{ |z| \leq \delta \}$ is replaced by a second order operator (resulting to a non-degenerate equation), and $u^\delta$ is in turn  approximated using a finite difference scheme by $u^{\delta, h, \tau}$, where $h$ and $\tau$ are the spatial and temporal discretization parameters respectively. This leads to estimates of the form  $\|u-u^{\delta,h,\tau}\| \leq N f(\delta)+ N(\delta)(\sqrt{\tau} +h)$ where $f(\delta)=\int_{|z|\leq \delta} z^3 \nu(dz) / \int_{|z|\leq \delta} z^2 \nu(dz)$. In this estimate, the constant $N(\delta)$ depends on $\delta$ and blows up as $\delta \to 0$ at a rate of $\nu(\{|z|\geq \delta\}) $, which is a consequence of the fact that the integro-differential operator is treated as a first/zero order operator away from the ball $(-\delta,\delta)$. In a similar manner in \cite{MIJ}, $\delta$ is a function of $h$, and the corresponding convergence rate for the spatial approximation  is of order $h\kappa (h/2)$ where $\kappa(\delta):=\int_{(-1,1) \setminus (-\delta,\delta)} |z| \nu(dz)$. If then for example the L\'evy measure has a density of the form $|z|^{-(2+\alpha)}$ for some $\alpha \in (0,1)$, then the convergence is of order $h^{(1-\alpha)}$, which can be very slow, depending on $\alpha$. The approach in \cite{DeF} is also similar (truncation of the integro-differential operator near zero). Under some technical conditions posed on the L\'evy measure (it is assumed to have a density of a particular form, that is twice continuously differentiable and has a prescribed behavior near zero), similar estimates are obtained (with constants blowing up as the truncation parameter $\delta \to 0$). 

In contrast to these works, in the present article we do not truncate the operator near the origin. We introduce an approximation that treats the integro-differential operator as a second order operator on the whole unit ball. Our approximation is similar to the one that we introduced in \cite{KD}, \cite{DL}. However, in these works the results and their proofs rely  on the non-degeneracy of the second order differential operator. We show that the approximate operator $J^h$ that we suggest here  is negative semi-definite,  and this combined with estimates obtained in \cite{IG} for the difference operators lead to apriori estimates of the solution of the scheme independent of the discretization parameters, without posing a non-degeneracy condition. This, combined with consistency estimates for the operators lead to estimates of the form
$\|u-u^{h,\tau}\| \leq N (h+\sqrt{\tau})$ where $N$ depends only on the data of the equation. We also show, under some more spatial regularity of the data, that  $\|u-u^{h,\tau}\| \leq N (h+\tau)$.  Also, let us note here that  we do not pose any additional assumption on the L\'evy measure $\nu$. 

The analysis of the spatial approximation is done in the spirit of \cite{HY}. The equations are first discretized in space  and  solved as equations in Sobolev spaces over $\bR$ ($u^{h}$)  and  as equations on the grid ($v^{h}$). 
Error estimates are obtained in Sobolev norms for the difference $u-u^{h}$. By embedding theorems,  the restriction of $u^{h}$ on the grid   is  shown to agree with $v^{h}$.  Hence, the error estimates in Sobolev norm imply pointwise error estimates for the difference $u-v^{h}$, by virtue of Sobolev embedding theorems. The discretized equations are further discretized in time (see also \cite{MaGy}), they are solved in Sobolev spaces ($u^{h,\tau}$) and on the grid ($v^{h,\tau}$), and estimates are obtained for $u^h-u^{h,\tau}$,  which in turn imply estimates for $v^h-v^{h,\tau}$. 

For degenerate equations not involving non-local operators we refer to \cite{GK} ,  \cite{GFD}, \cite{IG} and \cite{GG} where acceleration is also obtained in the convergence with respect to the spatial discretization parameter by means of Richardson's extrapolation. In the last three articles the results are obtained in a more general, stochastic setting, but the results remain optimal for deterministic equations as well.

In conclusion let us introduce some notation. By $u_t(x)$ we denote the value of a function $u:[0,T] \times \bR \to \bR$ at $(t,x)\in [0,T] \times \bR $ and when $u$ is understood as a function of $t$ with values in some function space (function of $x\in \bR$) we will write  $u_t:=u_t(\cdot)$ for $t\in [0,T]$.   By $\D_x $ we denote the derivative operator with respect to the spatial variable. The notation $C^\infty_c$ stands for the set of all smooth,  compactly supported,  real functions on $\bR$. We denote by $(\cdot,\cdot)$ and $\|\cdot\|_{L_2}$ the inner product and the norm respectively in $L_2(\bR)$. For an integer $l\geq 0$ , $H^l$ will be the Sobolev space of all function in $L_2(\bR)$ having distributional derivatives up to order $l$ in $L_2(\bR)$, with the inner product
$$
(f,g)_l =\sum_{j=0}^l (\D^j_x f, \D^j_x g),
$$
and we denote the corresponding norm by $\|\cdot\|_{H^l}$. For real number $\alpha, \beta$, we use the notation  $\alpha \wedge \beta:= \min\{\alpha, \beta\}$.  We use the notation $N$ for constants that may change from line to line. In the proofs of lemmas/theorems, the dependence of $N$ to certain parameters is given at the statement of the corresponding lemma/theorem.

\section{Formulation of the main results}     \label{Formulation}

In this section we introduce our scheme and we state our main results. 
From now on we will use the following notations 
$$
 \mu_0 := \nu(\bR \setminus [-1,1]),  \ \mu_2:= \int_{|z|\leq1} z^2 \nu(dz). 
$$
\begin{assumption}   \label{as: bounded coef}
Let $m \geq 1$ be an integer. 
\begin{itemize}
\item[i)] The functions $a, b,c: [0,T] \times \bR \to \bR$  are measurable in $(t,x)$. The functions $b,c$ and the function $a$,  together with their spatial derivatives up to order $m$ and up to order $\max(m,2)$ respectively, are continuous in $x \in \bR$ and bounded  in magnitude by a constant $K$, uniformly in $t \in [0,T]$. 

\item[ii)] The initial condition $\psi$ belongs to $H^m$ and $f:[0,T] \to H^m$ is a measurable function such that 
$$
\mathcal{K}_m^2=\|\psi\|_{H^m}^2+\int_0^T \|f_t\|^2_{H^m} dt < \infty. 
$$
\end{itemize}
\end{assumption}

\begin{assumption}                \label{as: elipticity}
For all $(t,x) \in [0,T] \times \bR$, we have $a_t(x) \geq 0$.
\end{assumption}

Notice that for $ \phi, \varphi \in C^\infty_c$, by virtue of Taylor's formula and integration by parts  we have
\begin{align*}
(J\phi, \varphi) =& -\int_{|z|\leq 1} \int_0^1 (1-\theta) z^2 (\D_x \phi(\cdot+\theta z), \D_x \varphi ) d\theta \nu(dz) \\
& +\int_{|z|>1} (\phi(\cdot+z)+ \phi, \varphi) \nu(dz).
\end{align*}
The solution of \eqref{eq: main equation}-\eqref{eq: main equation initial condition} is understood in the following sense.
\begin{definition}                      \label{de: solution main equation}
An $H^1$-valued weakly continuous function $(u_t)_{t \in [0,T]}$ is a solution to \eqref{eq: main equation}-\eqref{eq: main equation initial condition} if for all $\phi \in C^\infty_c$
\begin{align*}
(u_t,\phi)=(\psi, \phi)&+\int_0^t (\D_xu_s, -\phi \D_x a_t-a_s\D_x\phi+b_s\phi )+(c_s u_s , \phi) \ ds \\
&  -\int_0^t \int_{|z|\leq 1} \int_0^1 (1-\theta) z^2 (\D_x u_s (\cdot+\theta z), \D_x \phi ) d\theta \nu(dz) ds \\
& +\int_0^t \int_{|z|>1} (u_s(\cdot+z)+ u_s, \phi) \nu(dz) ds.
\end{align*}
\end{definition}
The following well-posedness result can be found in \cite{KD1} and \cite{ML}. 

\begin{theorem}
Let Assumptions \ref{as: bounded coef} and \ref{as: elipticity} hold. Then \eqref{eq: main equation}-\eqref{eq: main equation initial condition} has a unique solution $u:[0,T]\to H^1$. Moreover,  $u_t$ belongs to  $H^m$ for all $t\in [0,T]$, it is weakly continuous as $H^m$-valued function, strongly continuous as function with values $H^{m-1}$, and the following estimate holds
$$
\sup_{t \leq T}\|u_t \|^2_{H^m} \leq N \mathcal{K}^2_m,
$$
where $N$ is a constant depending only on $T$, $m$, $K$, $\mu_0$ and $\mu_2$. 
\end{theorem}
\begin{remark}
If Assumption \ref{as: bounded coef} holds with $m \geq 2$ in the above theorem, then the solution is strongly continuous $H^1$ valued function, which by the continuous  embedding $H^1 \hookrightarrow C^{0, 1/2}$ (space of bounded $1/2$-H\"older  continuous functions with the usual norm) implies that the solution $u_t(x)$ is a continuous function of $(t,x) \in [0,T] \times \bR$. 
\end{remark}

For $\lambda \in \bR \setminus \{0\}$  we define the following operators
 $$
\delta_\lambda \phi(x):=\frac{\phi(x+\lambda)-\phi(x)}{\lambda},  
\ \delta^\lambda \phi(x):= \frac{(\delta_\lambda+\delta_{-\lambda})\phi(x)}{2}. 
$$
We continue with the approximation of the integro-differential operator. 
For $h \in (0,1)$ we will denote our grid by $\mathbb{G}_h:=h \mathbb{Z}$, and for  integers $k  \geq 1$ we define
$$
B^h_k:=((k-1)h,kh],
$$
while for  integers $k \leq -1$ we define 
$$
B^h_k:= [kh, (k+1) h ).
$$
Notice that $B^h_0$ is not defined. From now on we assume that $h\in \{ 1/n : n \in \mathbb{N}_+\}=: \mathfrak{N}$.  We set $\bA_h:= \{ m \in \bZ : |m| \leq 1/h , m \neq 0 \}$ and $\bB_h:= \bZ \setminus( \bA_h \cup \{0\})$. 
Let us define the operators 
\begin{align*}
J_1^h \phi(x) &:= \sum_{k \in \bA_h} \zeta^h_k \sum_{l=0}^{|k|-1} \theta_k^l  \delta_{-h} \delta_h \phi(x+s_k l h ),
\\
J_2^h \phi (x)&:= \sum_{k \in \bB_h} \left(\phi (x+hk)-\phi (x) \right)\nu(B^h_k), 
\end{align*}
where 
$$
s_k = \frac{k}{|k|}, \ \zeta^h_k:= \int_{B^h_k} z^2 \ \nu(dz), \ \theta_k^l := \int_{l/|k|}^{(l+1)/|k|} (1-\theta) \ d \theta.
$$
We denote $J^h:=J^h_1+J^h_2$. The differential operator $L_t$ is approximated by $L^h_t$, given by
$$
L_t^h\phi(x):=a_t (x)\delta^h\delta^h\phi(x)  +b_t(x)\delta^h\phi(x) +c_t\phi(x). 
$$

We will write $l_2(\mathbb{G}_h)$ for the set of all real valued function  $\phi$ on $\mathbb{G}_h$ such that 
$$
\|\phi\|^2_{l_2(\mathbb{G}_h)}:= h \sum_{x\in \mathbb{G}_h} |\phi(x)|^2< \infty.
$$
We will denote the corresponding inner product by $(\cdot, \cdot)_{l_2(\mathbb{G}_h)}$. 
Let us now consider in $l_2(\mathbb{G}_h)$ the scheme
\begin{align}       \label{eq: discr v}
d v^h_t&= \left( (L^h_t+I^h) v^h_t + f_t \right) dt  \\   \label{eq: discr  v init con}
v^h_0&= \psi.
\end{align}
\begin{remark}
For $l\geq 1$ we have the continuous  embedding $H^l \hookrightarrow l_2(\mathbb{G}_h)$ (see \cite{IG}). Therefore under Assumption \ref{as: bounded coef} we have
$$
\|\phi\|^2_{l_2(\mathbb{G}_h)}+\int_0^T \|f_t\|^2_{l_2(\mathbb{G}_h)} dt < \infty.
$$
Under the same assumption it is easy to see that $L_t^h+J$ is a bounded linear operator on $l_2(\mathbb{G}_h)$ into itself (with norm bounded by a constant uniformly in $t \in[0,T]$).  Hence under Assumption \ref{as: bounded coef}, \eqref{eq: discr v}-\eqref{eq: discr  v init con} has a unique solution, that is,  a continuous function $v:[0,T] \to l_2(\mathbb{G}_h)$ such that for all $t \in [0,T]$
$$
v^h_t=\phi+\int_0^t (L_t^h+J)v^h_s+f_s \ ds,
$$
where the equality is understood in $l_2(\mathbb{G}_h)$ (hence,  also for all $x \in \mathbb{G}_h$). 
\end{remark}

Next is our main result concerning the spatial approximation.
\begin{theorem}                 \label{thm: main theorem space}
Let Assumptions \ref{as: bounded coef} and \ref{as: elipticity} hold with  $m\geq 4$. Let $u$ and $v^h$ be the unique solutions of \eqref{eq: main equation}-\eqref{eq: main equation initial condition} and \eqref{eq: discr v}-\eqref{eq: discr  v init con} respectively.  The following estimate holds,
$$
\sup_{t \in [0,T]} \sup_{x \in \mathbb{G}_h} |u_t(x)-v^h_t(x)|^2 + \sup_{t\in [0,T]}\|u_t-v^h_t\|^2_{l_2(\mathbb{G}_h)}  \leq N h^2 \mathcal{K}_m^2,
$$
where $N$ is a constant depending only on $m, K, \mu_0, \mu_2$ and $T$. 
\end{theorem}
We now move to the temporal discretization. 
Let $n\geq 1$ be an integer and let $\tau=T/n$. 
In $l_2(\mathbb{G}_h)$ we consider the implicit scheme
\begin{align}     \label{eq: space time dis v}
v_i&=v_{i-1}+ \tau [(L^h_{i\tau}+J^h)v_i +f_{i\tau} ] , \  i=1,...,n \\           \label{eq: space time dis init v}
v_0&= \psi.
\end{align}
\begin{theorem}                   \label{thm: existence uniqueness l2}
Let Assumptions \ref{as: bounded coef} and \ref{as: elipticity} hold.   There exists a constant $N_0$ depending only on $K$ and $T$, such that  for any $h \in \mathfrak{N}$,  if $n > N_0$, then \eqref{eq: space time dis v}-\eqref{eq: space time dis init v} has a unique solution $(v^{h,\tau})_{i=0}^n$. 
\end{theorem}

\begin{assumption}                  \label{as: continuity in time}
Let $l\geq 0 $ be an integer. There exist constants $C$ and $\gamma>0$  such that 
$$
|\D^j_x a_t(x)-\D^j_xa_s(x)|^2+|\D^j_xb_t(x)-\D^j_xb_s(x)|^2+|\D^j_xc_t(x)-\D^j_xc_s(x)|^2\leq C|t-s|^\gamma
$$
and 
$$
 \ \|f_t-f_s\|^2_{H^l}\leq C |t-s|^\gamma
$$
for all $x\in \bR$, $t,s \in [0,T]$, and $0\leq j \leq l$. 
\end{assumption}

\begin{assumption}                 \label{as: boundedness f}
There exists a constant $K'$, such that for all $t \in [0, T]$ we have $\|f_t\|_{H^{m-2}}^2 \leq K'$. 
\end{assumption}

Next is our result concerning the temporal approximation. 

\begin{theorem}                     \label{thm: main theorem}
Let Assumptions \ref{as: bounded coef}, \ref{as: elipticity} and \ref{as: boundedness f} hold 
with $m \geq 4$ and let Assumption \ref{as: continuity in time} hold with $l\geq 1$. 
 Let $(v^h_t)_{t\in[0,T]}$ and $(v^{h,\tau}_i)_{ i=0}^n$ be the unique solutions of equations \eqref{eq: discr v}-\eqref{eq: discr  v init con} and \eqref{eq: space time dis v}-\eqref{eq: space time dis init v} respectively (for $n>N_0$). 
There exists a constant $N_0'$ such that if $n>N_0'$,  then:

\begin{itemize}
\item[(i)]  the following estimate holds,
\begin{equation}      \nonumber
\max_{i \leq n} \sup_{x \in \mathbb{G}_h} |v^h_{i\tau}(x)-v^{h,\tau}_i(x)|^2  + \max_{i \leq n}\| v^h_{i\tau}-v^{h,\tau}_i\|^2_{l_2(\mathbb{G}_h)}  \leq \tau ^{1\wedge \gamma}  N (K'+\mathcal{K}^2_m)
\end{equation} 

\item[(ii)] if moreover $m \geq 5$,  then
\begin{equation}                    \nonumber           
\max_{i \leq n}  \sup_{x\in \mathbb{G}_h}|v^h_{i\tau}(x)-v^{h,\tau}_i(x)|^2+\max_{i \leq n}\| v^h_{i\tau}-v^{h,\tau}_i\|^2_{l_2(\mathbb{G}_h)} \leq \tau^{2\wedge \gamma} N (K'+\mathcal{K}^2_m),   
\end{equation}
\end{itemize}
where $N$ is a constant depending only on $K, C$, $T$, $m$, $\mu_0$ and $\mu_2$.  
\end{theorem}

A direct consequence of the theorem above is the following:

\begin{theorem}
Under the assumptions of Theorem \ref{thm: main theorem}, for all $n>N_0'$ and all $h \in \mathfrak{N}$, 
\begin{itemize}
\item[(i)]  the following estimate holds,
\begin{equation}      \nonumber
\max_{i \leq n} \sup_{x \in \mathbb{G}_h} |u_{i\tau}(x)-v^{h,\tau}_i(x)|^2  + \max_{i \leq n}\| u_{i\tau}-v^{h,\tau}_i\|^2_{l_2(\mathbb{G}_h)}  \leq N(h^2+\tau ^{1\wedge \gamma} )  \mathcal{N}^2_m
\end{equation} 

\item[(ii)] if moreover $m \geq 5$,  then
\begin{equation}                    \nonumber           
\max_{i \leq n}  \sup_{x\in \mathbb{G}_h}|u_{i\tau}(x)-v^{h,\tau}_i(x)|^2+\max_{i \leq n}\| u_{i\tau}-v^{h,\tau}_i\|^2_{l_2(\mathbb{G}_h)} \leq N(h^2+\tau ^{2\wedge \gamma} )  \mathcal{N}^2_m,
\end{equation}
\end{itemize}
where $\mathcal{N}^2_m=K'+\mathcal{K}^2_m$, and $N$ is a constant depending only on $K, C$, $T$, $m$, $ \mu_0$ and $\mu_2$.  
\end{theorem}

\section{Auxiliary Facts}     \label{Auxiliary}
In this section we prove some results that will be used in order to prove the main theorems.
\begin{lemma}                    \label{lem: non positive operator}
For any integer $l\geq 0 $, and for any $\phi \in H^l$,  we have 
$$
(\D^j_x J^h \phi, \D^j_x \phi ) \leq 0,
$$
for all integers $j \in \{0,...,l\}$. 
\end{lemma}
\begin{proof} 
Since $\D_x J^h \phi= J^h \D_x \phi$, it clearly suffices to show the conclusion with $l=j=0$.
We have
\begin{align*}
(J^h_2\phi(x), \phi(x)) &= \sum_{k \in \bB_h} \left((\phi(\cdot+hk), \phi)-\|\phi\|_{L_2}^2 \right) \nu(B^h_k) \\
&\leq  \sum_{k \in \bB_h} \left(\|\phi\|_{L_2}^2 -\|\phi\|_{L_2}^2 \right) \nu(B^h_k) =0,
\end{align*}
where the inequality is due to H\"older's inequality and the translation invariance of the Lebesgue measure. In order to show that $(J^h_1\phi , \phi) \leq 0$, 
clearly it suffices to show that for each $k \in \mathbb{A}_h$ 
\begin{equation}                       \label{eq: neg def}
\left( \sum_{l=0}^{|k|-1} \theta_k^l  \delta_{-h} \delta_h \phi(x+s_k l h ), \phi(x) \right) \leq 0.
\end{equation}
If $s_k=1$, then a simple calculation shows that
$$
\sum_{l=0}^{|k|-1} \theta_k^l  \delta_{-h} \delta_h \phi(x+s_k l h )
$$
\begin{align*}
&=\sum_{l=0}^{|k|-1} \frac{2|k|-(2l+1)}{2k^2h^2}\left[ \phi(x+(l-1)h)-2\phi(x+lh) + \phi(x+(l+1)h) \right] \\
&=\frac{1}{2k^2h^2}\left( \phi(x+kh)+\phi(x+(k-1)h)+(2k-1)\phi(x-h)-(2k+1)\phi(x) \right),
\end{align*}
which combined with H\"older's inequality  imply \eqref{eq: neg def}.  If $s_k=-1$, then 
$$
\sum_{l=0}^{|k|-1} \theta_k^l  \delta_{-h} \delta_h \phi(x+s_k l h )
$$
\begin{align*}
&=\sum_{l=0}^{|k|-1} \frac{2|k|-(2l+1)}{2k^2 h^2}\left[ \phi(x-(l+1)h)-2\phi(x-lh) + \phi(x-(l-1)h) \right]
\\
&=\frac{1}{2k^2h^2}\left( \phi(x+kh)+\phi(x+(k+1)h)+(2|k|-1)\phi(x+h)-(2|k|+1)\phi(x) \right),
\end{align*}
which again by virtue of H\"older's inequality implies \eqref{eq: neg def}.
\end{proof}
Lemma \ref{lem: non positive operator} combined with Lemma 3.4 from \cite{IG} imply the following: 

\begin{lemma}       \label{lem: coercivity}
Suppose Assumption \ref{as: bounded coef} (i) holds. Then for any  integer  $l \in \{0,..,m\}  $ and any $ \phi \in H^m$, we have
$$
(\D^l_x (L^h_t+J^h )\phi, \D^l_x \phi) \leq N \|\phi\|^2_{H^m},
$$
where $N$ is a constant depending only on $K$ and $m$. 
\end{lemma}
The following is very well known (see e.g. \cite{IG}, \cite{GYKR}). 
\begin{lemma}                            \label{lem: estimate derivatives}
For each  integer $l \geq 0$, there is a constant $N$ depending only on $l$,  such that for all $u \in H^{l+2}$, $v \in H^{l+3}$ and $\lambda \in \bR \setminus\{0\}$,
\begin{align*}
\|\delta^{\lambda}u-\partial_x u\|_{H^l}+\|\delta_{\lambda}u-\partial_x u\|_{H^l} & \leq N|\lambda| \|u\|_{H^{l+2}}, \\
\|\delta^\lambda\delta^\lambda v-\partial_x^2 v\|_{H^l} +\|\delta_{\lambda}\delta_{-\lambda} v-\partial_x^2 v\|_{H^l} &\leq N|\lambda| \|v\|_{H^{l+3}}.
\end{align*}  
\end{lemma}
For our approximation we have the following consistency estimates.
\begin{lemma}                      \label{lem: estimate J}
For any integer $l \geq 0$,  and any $\phi \in H^{l+3}$  we have
\begin{equation}           \label{eq: estimate J}
\|J^h\phi -J\phi \|_{H^l} \leq N h \|\phi\|_{l+3},
\end{equation}
where $N$ is a constant depending only on $l$, $\mu_0$ and $\mu_2$. 
\end{lemma}
\begin{proof}
Again we can and we will assume that $l=0$. 
We have 
\begin{align*}
J^h_2 \phi (x)- J_2\phi(x) & = \sum_{k \in \bB_h} \int_{B^h_k} \left( \phi(x+hk)-\phi(x+z) \right) \nu(dz) \\
&=\sum_{k \in \bB_h}  \int_{B^h_k} \int_0^1 (hk-z) \D_x \phi(x+z+\theta(hk- z) d\theta \nu(dz), 
\end{align*}
which combined with the fact that $|hk-z| \leq h$ for $z \in B^h_k$, gives by virtue of Minkowski's integral inequality
\begin{equation}                                          \label{eq: J2 estimate}
\|J^h_2 \phi- J_2\phi\| \leq h \mu_0 \|\phi\|_{H^1}  
\end{equation}
For $J_1^h-J_1$ we have
\begin{align}                      \label{eq: est integrand}
\nonumber 
&J^h_1\phi -J_1\phi \\  \nonumber 
=& \sum_{k \in \bA_h} \zeta^h_k \sum_{l=0}^{|k|-1} \theta_k^l  \delta_{-h} \delta_h \phi(x+s_k l h ) - \int_{|z|\leq 1} \int_0^1 (1-\theta) z^2\D^2_x \phi (x+\theta z)  d\theta \nu(dz) \\
=& \sum_{k \in \bA_h} \int_{B^h_k} \sum_{l=1}^{|k|-1} \int_{l/|k|} ^{(l+1)/|k|} z^2(1-\theta) \left(\delta_{-h} \delta_h \phi(x+s_klh) -\D^2_x \phi(x+\theta z)\right)  d\theta \nu(dz).
\end{align}
Then we have for the integrand in the above quantity
\begin{align}             \nonumber
 &\delta_{-h} \delta_h \phi(x+s_klh) -\D^2_x \phi(x+\theta z) \\
\nonumber 
=& \delta_{-h} \delta_h \phi(x+s_klh)-\delta_{-h} \delta_h \phi(x+\theta z ) 
+\delta_{-h} \delta_h \phi(x+\theta z )-\D^2_x \phi(x+\theta z) \\
   \nonumber
=& \int_0^1 (s_k l h - \theta z) \delta_{-h} \delta_h \D_x  \phi(x+\theta z+\eta (s_k l h - \theta z )) d \eta\\
\label{eq: integrand estimate}
&+  \delta_{-h} \delta_h \phi(x+\theta z )-\D^2_x \phi(x+\theta z).
\end{align}
Notice that for $\theta \in [l/|k|,  (l+1)/|k|)$ and for $z \in B^h_k$ we have
$$
|s_k l h - \theta z| \leq |s_klh-  \theta kh |+|\theta k h - \theta z| \leq 2 h.
$$ 
Hence, for the first term at the right hand side of \eqref{eq: integrand estimate} we have 
$$
\big\| \int_0^1 (s_k l h - \theta z) \delta_{-h} \delta_h \D_x  \phi(\cdot+\theta z \eta (s_k l h - \theta z )) d \eta \big\|_{L_2} \leq 2 h \|\phi\|_{H^3},
$$
while for the second one we have by Lemma \ref{lem: estimate derivatives}
$$
 \|\delta_{-h} \delta_h \phi(\cdot+\theta z )-\D^2_x \phi(\cdot+\theta z)\|_{L_2} \leq h \|\phi\|_{H^3}.
$$
Therefore, 
$$
 \|\delta_{-h} \delta_h \phi(\cdot +s_klh) -\D^2_x \phi(\cdot+\theta z)\|_{L_2} \leq N h \|\phi\|_{H^3}, 
$$
which combined with \ref{eq: est integrand} and Minkowski's inequality gives
$$
\|J^h_1\phi -J_1\phi\|_{L_2} \leq N h \|\phi\|_{H^3}. 
$$
By this inequality and  \eqref{eq: J2 estimate} we obtain \eqref{eq: estimate J}. 
\end{proof}

\begin{lemma}                       \label{lem: boundedness of L^h J^h}
Let (i) from Assumption \ref{as: bounded coef}  hold. Then for any $l \leq m$ and for any $\phi \in H^{l+2}$, $t \in [0,T]$ we have
\begin{align*}
\|L^h_t \phi \|^2_{H^l} +\|J^h \phi \|^2_{H^l}& \leq N \|\phi\|^2_{H_{l+2}},
\end{align*}
where $N$ is a constant depending only on $K,m$, $\mu_0$ and $\mu_2$. 
\end{lemma}
\begin{proof}
Clearly it suffices to show the inequality for $\phi \in C^\infty_c$. We have for $\lambda \neq 0$
$$
\delta_\lambda \phi(x)=\int_0^1 \D_x \phi(x+\theta \lambda ) d\theta .
$$
Hence by Minkowski's inequality we get $\| \delta_\lambda \phi \|_{L_2} \leq \|\D_x \phi\|_{L_2}$, which implies 
$$
\|\delta^\lambda \phi \|_{L_2} \leq N\|\phi\|_{H^1},  \ \|\delta^\lambda \delta^\lambda \phi \|_{L_2} \leq N\|\phi\|_{H^2},   \|\delta_\lambda \delta_{-\lambda }\phi \|_{L_2} \leq N\|\phi\|_{H^2}.
$$
Hence, by Assumption \ref{as: bounded coef} (i)  we have
$$
\|L^h_t \phi\|^2_{H^l} \leq N \|\phi \|^2_{H^{l+2}}. 
$$
By Minkowski's inequality we have
$$
\|J^h_1\phi\|_{L_2}  \leq  \sum_{k \in \bA_h} \zeta^h_k \sum_{l=0}^{|k|-1} \theta_k^l  \| \delta_{-h} \delta_h \phi( \cdot +s_k l h ) \|_{L_2} \leq \frac{1}{2} \mu_2 \|\phi\|_{H^2}
$$
and
$$
\| J_2^h \phi \|_{L_2} \leq  \sum_{k \in \bB^h} \| \phi(x+h k) - \phi\|_{L_2}  \leq 2\mu_0  \|\phi\|_{L_2}.
$$
These estimates,  combined with the fact that $\D_xJ^h=J^h\D_x$, give
$$
\|J^h\phi\|_{H^l} \leq N\|\phi\|_{H^{l+2}}.
$$
This finishes the proof.
\end{proof}
Next we consider in $L_2(\bR)$ the following scheme
\begin{align}       \label{eq: discr}
d u^h_t&= \left( (L^h_t+I^h) u^h_t + f_t \right) dt  \\   \label{eq: discr init con}
u^h_0&= \psi.
\end{align}

\begin{lemma}                \label{lem: consistency}
Let Assumption \ref{as: bounded coef} hold with some integer $l\geq 1$ instead of $m$.  Then \eqref{eq: discr}-\eqref{eq: discr init con}  has a unique $L_2$-solution $(u^h_t)_{t \in [0,T]}$ which is a continuous $H^l$-valued function. If moreover Assumption \ref{as: elipticity} holds, then there exists a constant $N=N(l,T,K)$ such that for all $h \in \mathfrak{N}$ 
\begin{equation}            \label{eq: main estimate}
\sup_{t \leq T}\|u^h_t\|^2_{H^l} \leq N  \mathcal{K}^2_l.
\end{equation}
\end{lemma}
\begin{proof}
Equation  \eqref{eq: discr}-\eqref{eq: discr init con} is a differential equation on $L_2$ with Lipschitz continuous coefficients and therefore has a unique $L_2$-valued continuous solution $(u^h_t)_{t \in [0,T]}$. Similarly, it is a differential equation on $H^l$ with Lipschitz continuous coefficients and therefore has a unique $H^l$-valued continuous  solution $(w^h_t)_{t \in [0,T]}$. Since $H^l \subset L_2$ we have that $w^h=u^h$. 

For  \eqref{eq: main estimate},  we have for any $t \in [0,T]$
\begin{align*}
\|u^h_t\|^2_{H_l}&=\|\psi\|^2_{H^l} + \int_0^t \left[\left( (L^h_s +I^h) u^h_s, u^h_s\right)_{H^l}+ (f_s, u^h_s)_{H^l} \right] \ ds 
\\ 
&\leq  \|\psi\|^2_{H^l} + N\int_0^t \|u^h_s \|^2_{H^l}  \ ds+\int_0^T\|f_s\|^2_{H^l} \  ds< \infty,
\end{align*}
where the last inequality is by virtue of   Lemma \ref{lem: coercivity} and Young's inequality. Gronwall's lemma finishes the proof.
\end{proof}

\begin{theorem}                   \label{thm: spatial error Sobolev}
Let Assumptions \ref{as: bounded coef} and \ref{as: elipticity} with $m \geq 4 $, and let $u^h$ and $u$ be the unique solutions of  \eqref{eq: discr}-\eqref{eq: discr init con} and  \eqref{eq: main equation}-\eqref{eq: main equation initial condition} respectively. Then for any $h \in \mathfrak{N}$ we have
\begin{equation}
\sup_{t \leq T} \|u_t-u^h_t\|^2_{H^{m-3}} \leq N  \mathcal{K}^2_m h  ,
\end{equation}
where $N$ is a constant depending only on $m,T,\mu_0$, $\mu_2$ and $K$. 
\end{theorem}
\begin{proof}
We have that $u^h-u$ satisfies the conditions of Lemma \ref{lem: consistency} with $l=m-3$, $\psi=0$ and $f_t=(L^h_t-L_t)u_t+(I^h-I)u_t$. Therefore we have
\begin{align}
\sup_{t \leq T} \|u^h_t-u_t \|^2_{H^{m-3}} \leq & N \int_0^T \|(L^h_t-L_t)u_t+(J^h-J)u_t\|^2_{H^{m-3}}  \ dt \\
\leq & N h \int_0^T \|u\|^2_{H^m} \ dt \leq N h \mathcal{K}_m^2. 
\end{align}
where the second inequality follows from Lemmata \ref{lem: estimate derivatives} and  \ref{lem: estimate J}. This finishes the proof. 
\end{proof}
Next we continue with the time discretization. Let us consider on $L_2(\bR)$ the following implicit scheme. 
\begin{align}     \label{eq: space time dis}
u_i&=u_{i-1}+ \tau [(L^h_{i\tau}+J^h)u_i +f_{i\tau} ] , \  i=1,...,n \\           \label{eq: space time dis init}
v_0&= \psi.
\end{align}

The following is very well known. 
\begin{lemma}        \label{lem: solvability hilbert}
Let $\mathbb{D}$ be a bounded linear operator on a Hilbert space $X$ into itself. If there exists $\delta>0$ such that $(\mathbb{D}\phi, \phi)_X \geq \delta \|\phi\|_X^2$, for all $\phi \in X$, then for any $f \in X$, there exists a unique $g \in X$ such that
$\mathbb{D}g=f$.
\end{lemma}

\begin{theorem}                               \label{thm: existence uniqueness L2}
Let Assumptions \ref{as: bounded coef} and \ref{as: elipticity} hold. Then there exists a constant $N'$ depending only on $K$, $T$ and $m$, such that if $n>N'$, for any $h \in \mathfrak{N}$ there exists a unique $L_2$-solution $(u_i^{h,\tau})_{i=0}^n$ of \eqref{eq: space time dis}-\eqref{eq: space time dis init}. Moreover $u^{h,\tau}_i \in H^m$ for each $i=0,...,n$.
\end{theorem}
\begin{proof}
Let us write \eqref{eq: space time dis} in the form 
$$
\mathbb{D}_i u_i=F_i, \ i=1,...,n, 
$$
where 
$$
\mathbb{D}_i= I-\tau (L^h_{i\tau}+J^h), \ F_i=v_{i-1}+f_{(i-1)\tau}. 
$$
For each $i=1,...,n$, $\mathbb{D}_i$ is a bounded linear operator from $H^k$ to $H^k$ for all $k=0,...,m$. By Lemma \ref{lem: coercivity} we have
$$
(\mathbb{D}_i \phi, \phi)_k=\|\phi\|^2_{H^k}-\tau\left((L^h_{i\tau}+J^h)\phi, \phi\right)_k \geq \|\phi\|^2_{H^k}-\tau N \|\phi\|^2_{H^k},
$$
for all $k=0,..,m$, with $N$ depending only on $K$ and $m$. 
Hence, if $n > T N $, then we have with $\lambda :=1-(\tau/N)>0$
$$
(\mathbb{D}_i \phi, \phi)_k\geq \lambda \|\phi\|^2_{H^k}.
$$
The conclusion follows from the lemma above.

\end{proof}

\begin{theorem}
Let Assumptions \ref{as: bounded coef}, \ref{as: elipticity} and \ref{as: boundedness f} hold 
with $m \geq 4$ and
 let $(u^h_t)_{t\in[0,T]}$ and $(u^{h,\tau}_i)_{i=0}^n$ be the unique solutions of equations \eqref{eq: discr}-\eqref{eq: discr init con} and \eqref{eq: space time dis}-\eqref{eq: space time dis init} respectively (for $n >N'$). There exists a constant $N_1$ such that if $n>N_1$, then:

\begin{itemize}
\item[(i)] if Assumption \ref{as: continuity in time} holds with $l=m-3$,   then
\begin{equation}      \label{eq: time error m-3}
\max_{i \leq n} \|u^h_{i\tau}-u^{h,\tau}_i\|^2_{H^{m-3}}  \leq \tau ^{1\wedge \gamma}  N (K'+\mathcal{K}^2_m)
\end{equation} 

\item[(ii)] if Assumption \ref{as: continuity in time} holds with $l=m-4$,  then
\begin{equation}                         \label{eq: time error m-4}
\max_{i \leq n} \|u^h_{i\tau}-u^{h,\tau}_i\|^2_{H^{m-4}}  \leq \tau^{2\wedge \gamma} N (K'+\mathcal{K}^2_m),   
\end{equation}
\end{itemize}
where $N$ is a constant depending only on $K, C$, $T$, $m$, $\mu_0$ and $\mu_2$.  
\end{theorem}
\begin{proof}
In order to ease the notation, let us introduce $e_i=u^h_{i\tau}-u^{h,\tau}_i$. We have that $(e_i)_{i=0}^n$ satisfies
\begin{align*}
e_i&=e_{i-1}+\tau  \mathbb{R}_i e_i+ \mathbf{F}_i, \ i=1,,,.n, \\
e_0&=0,
\end{align*}
where 
$$
\mathbb{R}_i= L^h_{i\tau}+J^h, \ \mathbf{F}_i:=\int_{(i-1)\tau}^\tau F_t  \ dt
$$
$$
\ \  F_t:=(L^h_t+J^h)u^h_t-(L^h_{k(t)}+J^h) u^h_{k(t)}+ f_t-f_{k(t)},
$$
and  $k(t)= \lfloor nt \rfloor /n$. 
By the identity $\|b\|^2-\|a\|^2=2(b,b-a)-\|b-a\|^2$, we have for $j \leq m-3$ and $i \geq 1$, 
\begin{equation}        \label{eq: error ito}
\|\D^j_x e_i\|^2-\|\D^j_x e_{i-1}\|^2\leq 2 \tau ( \D^j_x e_i, \mathbb{R}_i \D^j_x e_i ) + 2(\D^j_x e_i, \D^j_x \mathbf{F}_i)
\end{equation}
By Lemma \ref{lem: coercivity} we have 
$$
2 \tau ( \D^j_x e_i, \mathbb{R}_i \D^j_x e_i )  \leq  \tau N \|\D^j_x e_i\|^2,
$$
while by Young's inequality we have
\begin{align*}
2(\D^j_x e_i, \D^j_x \mathbf{F}_i)& \leq \tau\|\D^j_x e_i,\|^2+\tau^{-1}\|\int_{(i-1)\tau}^{i \tau} \D^j_x F_t dt\|^2 \\
& \leq \tau\|\D^j_x e_i,\|^2+\int_{(i-1)\tau}^{i \tau} \|\D^j_x F_t \|^2 \ dt. 
\end{align*}
By using these inequalities and summing up \eqref{eq: error ito} over $0\leq j\leq q$, where  $q \in \{m-4,m-3\}$,  and  over $i\leq \rho \leq n$, we get
$$
\|e_\rho \|^2_{H^q} \leq \tau N \sum_{i =1}^\rho\|e_i\|^2_{H^q}+ N \int_0^T \|F_t\|^2_{H^q} \ dt< \infty,
$$
where $N$ is a constant depending only on $m$ and $K$. Let us set $N_1:=TN$. By the discrete Gronwall inequality we have for $n>N_1$ (i.e.  for $\tau < 1/N$)
$$
\max_{\rho \leq n} \|e_\rho\|^2_{H^q} \leq N \int_0^T \|F_t\|^2_{H^q} \ dt,
$$
where $N$ depends only on $m, K$ and $T$. We only have to estimate the term at the right hand side of the above inequality.
\begin{align}        \nonumber 
\int_0^T \|F_t\|^2_{H^q} dt  \leq & N\int_0^T \|(L^h_t-L^h_{k(t)}) u^h_t \|^2_{H^q} dt  \\           \nonumber 
 +&N \int_0^T \| (J^h+L^h_{k(t)})(u^h_t-u^h_{k(t)})\|^2_{H^q}  dt \\
\label{eq: estimate F}
 +&N\int_0^T \|f_t-f_{k(t)}\|^2_{H^q} dt .       
\end{align}
Let us show first \eqref{eq: time error m-3} under Assumption \ref{as: continuity in time} with $l=m-3$.  By Assumption \ref{as: continuity in time} and \eqref{eq: main estimate} we have  with $q=m-3$
\begin{align}                        \label{eq: estimate L_t-L_k(t)}
\int_0^T \|(L^h_t-L^h_{k(t)}) u^h_t \|^2_{H^q} dt & \leq \tau^\gamma N  \int_0^T \|u^h_t\|^2_{H^{q+2}} dt \leq \tau^\gamma N  \mathcal{K}^2_{q+2}  \\          \label{eq: estimate f_t-f_k(t)}
\int_0^T \|f_t-f_{k(t)}\|^2_{H^q} dt & \leq  \tau ^\gamma  T .
\end{align}
By Lemma \ref{lem: boundedness of L^h J^h} we have
$$
\int_0^T \| (J^h+L^h_{k(t)})(u^h_t-u^h_{k(t)})\|^2_{H^q}  dt \leq N\int_0^T \|u^h_t-u^h_{k(t)}\|^2_{H^{q+2}} dt. 
$$
Therefore, in order to show $(i)$ we only need to show that 
\begin{equation}           \label{eq: estimate u_t-u_k(t)}
\int_0^T \|u^h_t-u^h_{k(t)}\|^2_{H^{m-1}} dt \leq N \tau  (\mathcal{K}^2_m +K').
\end{equation}
For $\phi \in H^{m-1}$, and $\phi' \in H^m$,  one has $|(\phi', \phi)_m| \leq \|\phi'\|_{H^m}\|\phi\|_{H^{m-2}}$. Using this  and Young's inequality we obtain for $s,t \in [0,T]$ with $s \leq t$
\begin{align*}           
\|u^h_t-u^h_s\|^2_{m-1}& =2\int_s^t \left(u^h_r-u^h_s, (L^h_r+J^h)u_r+ f_r \right)_{m-1}  dr \\
& \leq N \int_s^t \|u^h_r-u^h_s\|^2_{H^m}+\|(L^h_r+J^h)u^h_r\|^2_{H^{m-2}}+ \|f_r\|^2_{H^{m-2}} dr  \\
& \leq N \int_s^t \sup_{t'  \leq T } \|u^h_{t'}\|^2_{H^m} +\|f_r\|^2_{H^{m-2}} dr \\
& \leq N(\mathcal{K}_m^2+K')(t-s),
\end{align*}
where the last inequality follows by Lemma \ref{lem: consistency} and Assumption \ref{as: boundedness f}. This shows \eqref{eq: estimate u_t-u_k(t)}, which combined with \eqref{eq: estimate L_t-L_k(t)} and  \eqref{eq: estimate f_t-f_k(t)} (with $q=m-3$), imply \eqref{eq: time error m-3} by virtue of \eqref{eq: estimate F}. In order to show \eqref{eq: time error m-4} under Assumption \ref{as: continuity in time} with $l=m-4$, by virtue of \eqref{eq: estimate F}, \eqref{eq: estimate L_t-L_k(t)} and \eqref{eq: estimate f_t-f_k(t)}, with $q=m-4$, it suffices to show 
$$
\int_0^T \|u^h_t-u^h_{k(t)}\|^2_{H^{m-2}} dt \leq N \tau^2  (\mathcal{K}^2_m +K').
$$
For $t,s \in [0,T]$ we have
\begin{align*}
\|u^h_t-u^h_s\|^2_{H^{m-2}} & \leq  \big\|\int_s^t (L^h_r+J^h)u^h_r+f_r \ dr \big\|^2_{H^{m-2}} \\
& \leq \big( \int_s^t  N \sup_{t'\leq T } \|u^h_{t'}\|_{H^m}+\|f_r\|_{H^{m-2}} \ dr \big)^2 \\
& \leq N(t-s)^2(\mathcal{K}^2_m+K').
\end{align*}
This brings the proof to an end. 

\end{proof}

\section{Proofs of the main results}         \label{proofs}
We are now ready to prove the main theorems.
\begin{proof}[Proof of Theorem \ref{thm: main theorem space}]
Let $\mathfrak{I}$, $\mathfrak{K}$ denote the continuous embeddings $H^{m-3} \hookrightarrow l_2(\mathbb{G}_h)$ and $H^{m-3} \hookrightarrow C^{0,1/2}$. Let $u^h$ and $v^h$ denote the solutions of \eqref{eq: discr}-\eqref{eq: discr init con} and \eqref{eq: discr v}-\eqref{eq: discr  v init con} (the same equation,  considered on $l_2(\mathbb{G}_h)$ and $L_2(\mathbb{G}_h)$). By applying $\mathfrak{I}$ to both sides of \eqref{eq: discr} we see that
$\mathfrak{I}u^h$ satisfies \eqref{eq: discr v}-\eqref{eq: discr  v init con}. Therefore $\mathfrak{I}u^h=v^h$ by uniqueness. Notice also that $\mathfrak{K}u^h_t = \mathfrak{I}u^h_t$, and $u_t(x)=\mathfrak{I}u_t(x)=\mathfrak{K}u_t(x)$,  for all $t \in [0,T]$ and $x \in \mathbb{G}_h$. 
Hence 
\begin{align*}
\sup_{x \in \mathbb{G}_h} |v^h_t(x)-u_t(x)|& =\sup_{x \in \mathbb{G}_h} |\mathfrak{I}u^h_t(x)-u_t(x)| \\
&= \sup_{x \in\mathbb{G}_h} |\mathfrak{K}u^h_t(x)-\mathfrak{K} u_t(x)|  \\
& \leq N \|u^h_t-u_t\|_{H^{m-3}}, 
\end{align*}
and 
\begin{align*}
\| v^h_t-u_t\|_{l_2(\mathbb{G}_h)} &=\| \mathfrak{I}u^h_t-\mathfrak{I}u_t\|_{l_2(\mathbb{G}_h)} \\
& \leq  N \|u^h_t-u_t\|_{H^{m-3}}, 
\end{align*}
where $N$ depends only on $m$. The conclusion now  follows from Theorem \ref{thm: spatial error Sobolev}.
\end{proof}
We move to the proof of Theorem \ref{thm: existence uniqueness l2}. Notice that the existence part follows easily from Theorem \ref{thm: existence uniqueness L2}. Namely, if $u^{h,\tau}$ solves \eqref{eq: space time dis}-\eqref{eq: space time dis init}, then $\mathfrak{I}u^{h,\tau}$ solves \eqref{eq: space time dis v}-\eqref{eq: space time dis init v}. Also, the uniqueness part is immediate if for example one poses a CLF condition on $\tau$ and $h$. However such a condition is obviously not necessary, and therefore, in order to prove Theorem \ref{thm: existence uniqueness l2}, we will proceed as in the proof of Theorem \ref{thm: existence uniqueness L2}. Hence, we need the following, whose proof is essentially given in \cite{GFD}, but we give a sketch for the convenience of the reader.
\begin{lemma}                                \label{lem: coerc l2}
Let Assumptions \ref{as: bounded coef} and \ref{as: elipticity} hold with $m=1$. Then for any $\phi \in l_2(\mathbb{G}_h)$ we have
$$
\left((L^h_t+J^h)\phi, \phi \right)_{l_2(\mathbb{G}_h)} \leq N \|\phi\|^2_{l_2(\mathbb{G}_h)},
$$
where $N$ depends only on $K$. 
\end{lemma}
\begin{proof}
One can replace $(\cdot, \cdot)$ with $(\cdot, \cdot)_{l_2(\mathbb{G}_h)}$ in the proof of Lemma \ref{lem: non positive operator} to obtain
$$
(J^h\phi, \phi )_{l_2(\mathbb{G}_h)} \leq  0.
$$
Consequently we only need that $(L_t^h \phi, \phi)_{l_2(\mathbb{G}_h)} \leq N \|\phi\|^2_{l_2(\mathbb{G}_h)}$. This is proved in \cite{GFD}. In the proof of Lemma 3.3 in  that article,  one can replace $(\cdot, \cdot)$ with $(\cdot, \cdot)_{l_2(\mathbb{G}_h)}$ to obtain
\begin{equation}            \label{eq: corc in l2}
|(\delta^h \phi, (\delta^h a_t )T^h \phi)_{l_2(\mathbb{G}_h)}| +|(b_t \delta^h \phi, \phi)_{l_2(\mathbb{G}_h)}+|(c_t\phi, \phi)|_{l_2(\mathbb{G}_h)} \leq N \|\phi\|^2_{l_2(\mathbb{G}_h)}, 
\end{equation}
where $T^h\phi(x)= (\phi(x+h)+\phi(x-h))/2$, and $N$ depends only on $K$. 
It is shown also in \cite{GFD} (see (3.3)) that for functions $u,v$
$$
\delta^h(uv)=(\delta^hu)T^hv+(\delta^hv)T^hu. 
$$
Therefore, 
\begin{align} \nonumber
(a_t\delta^h \delta^hu_t,u_t)_{l_2(\mathbb{G}_h)}=&-( \delta^hu_t,\delta^h(a_tu_t))_{l_2(\mathbb{G}_h)} \\   \label{eq: addu,u}
=& - (\delta^hu_t,(\delta^ha_t)T^hu_t)_{l_2(\mathbb{G}_h)}-( \delta^hu_t,(T^ha_t)\delta^hu_t)_{l_2(\mathbb{G}_h)}.
\end{align}
Notice that by virtue of Assumption \ref{as: elipticity}, we have
$$
-( \delta^hu_t,(T^ha_t)\delta^hu_t)_{l_2(\mathbb{G}_h)} \leq 0.
$$
Hence, \eqref{eq: addu,u} and \eqref{eq: corc in l2} imply 
$$
(L^h_t\phi, \phi )_{l_2(\mathbb{G}_h)} \leq N \|\phi\|^2_{l_2(\mathbb{G}_h)}.
$$
\end{proof}

\begin{proof}[Proof of Theorem \ref{thm: existence uniqueness l2}]
The proof is the same as the one of Theorem \ref{thm: existence uniqueness l2}, this time using Lemma \ref{lem: coerc l2} instead  of Lemma \ref{lem: coercivity}.
\end{proof}

\begin{proof}[Proof of Theorem \ref{thm: main theorem}]
The conclusion follows by Sobolev embeddings and Theorem \ref{thm: existence uniqueness L2}, similarly to the proof of Theorem \ref{thm: main theorem space}.
\end{proof}

\end{document}